\documentclass{article}

\usepackage{graphicx}
\usepackage[utf8]{inputenc}
\usepackage{amsfonts, amsmath, amssymb, amsthm}
\usepackage[letterpaper, portrait, margin=1in]{geometry}
\usepackage{comment}
\usepackage{xcolor}
\usepackage[shortlabels]{enumitem}
\usepackage{appendix}

\newtheorem{theorem}{Theorem}
\newtheorem{lemma}{Lemma}
\newtheorem{proposition}{Proposition}
\newtheorem{corollary}{Corollary}

\newcommand{\deriv}[2]{\frac{d#1}{d#2}}
\newcommand{\pderiv}[2]{\frac{{\partial}#1}{{\partial}#2}}

\newcommand{\Gxx}[0]{G}
\newcommand{\pcos}[0]{P}
\newcommand{\ppr}[0]{P'}
\newcommand{\dpr}[0]{D}
\newcommand{\Gnx}[0]{G_{-}}
\newcommand{\Gpx}[0]{G_{+}}
\newcommand{\Gxn}[0]{G^{-}}
\newcommand{\Gxp}[0]{G^{+}}

\newcommand{\mapderiv}[1]{F'_{#1}}			
\newcommand{\mapuderiv}[1]{B'_{#1}}			
\newcommand{\flow}[1]{\varphi^{#1}}

\newcommand{\figCond}[0]{1}
\newcommand{\figParam}[0]{2}
\newcommand{\figdBdr}[0]{3}
\newcommand{\figWave}[0]{4}
\newcommand{\figlfixed}[0]{5}
\newcommand{\figlzero}[0]{6}

\newcommand{\keyterm}[1]{\textit{#1}}

\newcommand{\cvec}[2]{\begin{bmatrix} #1 \\ #2 \end{bmatrix}}
\newcommand{\rvec}[2]{\begin{bmatrix} #1 & #2 \end{bmatrix}}
\newcommand{\mat}[4]{\begin{bmatrix} #1 & #2 \\ #3 & #4 \end{bmatrix}}

\newcommand{\makenest}[3]{\newcommand{#1}[1]{\ensuremath{\left#2##1\right#3}}}
\makenest{\abs}||
\makenest{\norm}\lVert\rVert
\makenest{\parens}()
\makenest{\bracket}[]
\makenest{\set}\{\}
\makenest{\clop}[)
\makenest{\opcl}(]

\newcommand{\R}{\mathbb{R}}
\newcommand{\Z}{\mathbb{Z}}
\newcommand{\eps}{\varepsilon}

\begin{document}

\title{On the Density of Dispersing Billiard Systems with Singular Periodic Orbits}
\author{Otto Vaughn Osterman}
\date{October 22, 2019}

\maketitle

\begin{abstract}
Dynamical billiards, or the behavior of a particle traveling in a planar region $D$ undergoing elastic collisions with the boundary, has been extensively studied and is used to model many physical phenomena such as a Boltzmann gas.
Of particular interest are the dispersing billiards, where $D$ consists of a union of finitely many open convex regions. These billiard flows are known to be ergodic and to possess the $K$-property.
However, Turaev and Rom-Kedar (1998) proved that for dispersing systems permitting singular periodic orbits, there exists a family of smooth Hamiltonian flows with regions of stability near such orbits, converging to the billiard flow. They conjecture that systems possessing such singular periodic orbits are dense in the space of all dispersing billiard systems and remark that if this conjecture is true then every dispersing billiard system is arbitrarily close to a non-ergodic smooth Hamiltonian flow with regions of stability.
In this paper, we consider billiard tables consisting of the complement to a union of open unit disks with disjoint closures. We present a partial solution to this conjecture by showing that if the system possesses a near-singular periodic orbit satisfying certain conditions, then it can be perturbed to a system that permits a singular periodic orbit. We comment on the assumptions of our theorem that must be removed to prove the conjecture of Turaev and Rom-Kedar for these systems.
\end{abstract}


\section{Introduction}

The problem of modeling the behavior of a particle traveling in a closed region $D$ (the \keyterm{billiard table}) at a constant speed undergoing elastic collisions with its boundary is known as \keyterm{dynamical billiards}. This has been extensively studied and is used to model many physical systems such as the dynamics of gas particles. In addition, many problems can be reduced to billiard problems in higher dimensions such as the motion of $N$ rigid spheres in a box (Sinai and Chernov 1987).

Since the speed of an orbit within a billiard is constant, we consider the phase space as consisting of points $ (p,v) \in D \times S^1 $, with position $p$ and velocity $v$.
The billiard flow $ \flow{t} : D \times S^1 \rightarrow D \times S^1 $ returns the state of a particle after time $t$ as a function of its initial state.
We refer to the set of all phase values $(p,v)$ assumed from a given initial condition $(p_0,v_0)$ as an \keyterm{orbit}, and the path of a particle, or the set of all positions $p$ in an orbit, as a \keyterm{trajectory}. By a \keyterm{segment} of a trajectory, we mean a straight-line portion between two collision points.

Of special interest in dynamical billiards are the \keyterm{dispersing billiards}, where the billiard table $D$ is the complement of a union of finitely many open convex regions. It is well-known that dispersing billiards are \keyterm{ergodic} and possess the \keyterm{$K$-property} (Bunimovich and Sinai 1987, Sinai 1970), which is associated with strong statistical properties of the state of the system over time.

Due to its use in modeling gas particles and the strong statistical properties of $K$-systems, dispersing billiards have been a key model to study the main assumption of statistical mechanics, the Boltzmann ergodic hypothesis, which is concerned with the statistical behavior of large-scale systems of gas particles (Sz\'{a}sz 2000).

A \keyterm{Hamiltonian system} is a dynamical system governed by the equations
\begin{equation} \label{hamiltonian}
    q'(t) = \pderiv{H}{p}, \qquad
    p'(t) = -\pderiv{H}{q},
\end{equation}
where $ p, q \in \R^n $. Hamiltonian systems represent a general class of physical systems, where $q$ is the \keyterm{generalized position} and $p$ the \keyterm{generalized momentum}. The \keyterm{Hamiltonian} of the system $H(q,p,t)$ generally represents its total energy.
In \keyterm{automonous} Hamiltonian systems, where the Hamiltonian is independent of $t$, the Hamiltonian is a conserved quantity.

The billiard flow can be described as an autonomous Hamiltonian system with position $p$, momentum $v$, and Hamiltonian
\begin{equation}
    H(p,v) = \frac{1}{2} \norm{v}^2 + V(p),
\end{equation}
where the \keyterm{Hamiltonian potential} is
\begin{equation} \label{bPotential}
    V(p) = \begin{cases} 0, & p \in D \\ \infty, & p \notin D \end{cases}
\end{equation}
(Turaev and Rom-Kedar 1998). Note that the Hamiltonian in this case is not continuous on the boundary of $D$, which causes discontinuities in the velocity at collision points. However, in most physical applications, the Hamiltonian potential is smooth. Therefore, the billiard flow serves only an idealized model, and approximations by smooth potential functions would more accurately reflect most physical systems.

Turaev and Rom-Kedar (1998) consider perturbations of the billiard flow by replacing the billiard potential with smooth approximations. This produces smooth Hamiltonian flows that are continuous or smooth everywhere the billiard flow is continuous or smooth, respectively.

It is natural to ask what properties of dispersing billiard systems are preserved when they are purturbed in this manner. Rapoport et al. (2005) consider families of $C^r$ smooth Hamiltonian potentials whose limit is the non-smooth billiard potential. They prove that the corresponding Hamiltonian flows for time $T$ indeed converge to the billiard flow in the $C^r$ sense, provided there are finitely many collisions in the time interval $[0,T]$, all of which are non-singular.

Other types of perturbations of the billiard potential have also been studied (Baldwin 1988, Knauf 1989, Kubo 1976). For example, Baldwin (1988) introduced the concept of a \keyterm{soft billiard system}, with potential
\begin{equation} \label{softPotential}
    V(p) = \begin{cases} 0, & p \in D \\ U, & p \notin D, \end{cases}
\end{equation}
where $U$ is a constant. These flows approach the true billiard flow as $ U \rightarrow \frac{1}{2} $.

Turaev and Rom-Kedar (1998) address the ergodicity property of smooth Hamiltonian approximations by considering a family of such flows that converge to the billiard flow. Of particular importance in their analysis are \keyterm{singular periodic orbits}, or periodic orbits containing a collision with angle of incidence $ \pm \frac{\pi}{2} $ (such a collision is called \keyterm{grazing}).
Specifically, they prove that in dispersing billiard systems with a singular periodic orbit, there exists a family of smooth Hamiltonian flows converging to the billiard flow, all of which contain regions of stability near the singular periodic orbit.
As a consequence, such dispersing billiard systems lose ergodicity with these perturbations, as well as the strong statistical properties associated with $K$-systems.

Turaev and Rom-Kedar (1998) conjecture that systems possessing such singular periodic orbits are dense in the set of all dispersing billiards. In another paper (2012), they further investigated this conjecture and provided supporting numerical experiments. If this is true, it would imply that every dispersing billiard in the plane is arbitrarily close to a smooth Hamiltonian flow with regions of stability.
In this paper, we address this conjecture by proving that it is indeed true for a special case of a near-singular periodic orbit.

It is well-known that the set of periodic orbits in a dispersing billiard is dense in the phase space. Therefore, for any $ \delta > 0 $, there is a periodic orbit such that one segment of the trajectory passes a distance of less than $\delta$ from the boundary of a scatter (which we call scatterer $0$) without colliding with it. Therefore, to prove the conjecture, it is sufficient to show that for every $ \eps > 0 $, there exists $ \delta > 0 $ depending only on the system and $ \eps $ such that scatterer $0$ can be moved a distance of at most $\eps$, while allowing the periodic orbit to be perturbed continuously in a way so that either the trajectory eventually reaches scatterer $0$ with a grazing collision or a singularity occurs elsewhere in the trajectory.
We present a partial solution to this problem by proving the following:

\begin{theorem} \label{mainTh}
    Consider any billiard system in the plane $\R^2$ or the flat torus $\R^2/\Z^2$, where the billiard table $D$ consists of the complement of a union of finitely many open unit disks with disjoint closures, and any periodic orbit that satisfies the following two conditions:
    \begin{itemize}
        \item One segment of the trajectory passes a distance $ R < \delta $ from the boundary of one scatterer, which we refer to as scatterer $0$.
        \item There is exactly one collision with the boundary of scatterer $0$ per period. If $\alpha_0$ is angle of incidence of this collision, then $ \abs{\alpha_0} < \frac{\pi}{6} $.
    \end{itemize}
    Then, for any $ \eps > 0 $, there exist $ \delta > 0 $ depending only on the location of the scatterers and $ \alpha_0 $, such that for any periodic orbit introduced above, it is possible to perturb the billiard by moving scatterer $0$ a distance of at most $ \eps $ while perturbing the periodic trajectory continuously, so that the periodic orbit becomes singular.
\end{theorem}
The conditions for this theorem are demonstrated in Figure \figCond. If the second bulleted condition could be removed, it would imply that Turaev and Rom-Kedar's conjecture regarding the density of dispersing billiards with singular periodic orbits is true for all billiard tables consisting of the union of finitely many open unit disks with disjoint closures.

\begin{figure} \label{figCond}
    \includegraphics[width=84mm]{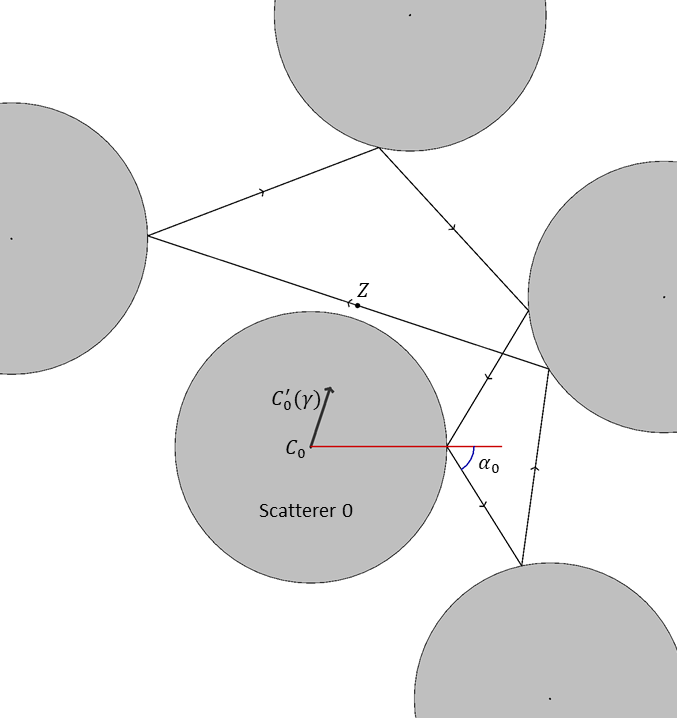}
    \caption{For Theorem \ref{mainTh}, we require that $ \abs{\alpha_0} < \frac{\pi}{6} $, and that there is only one collision point with scatterer $0$ per period. We then move scatterer $0$ toward point $Z$, assumed to be a distance of less than $\delta$ from scatterer $0$.}
\end{figure}

In Section \ref{SecMap} of this paper, we describe the problem in terms of a collision map. Then, in Section \ref{SecAnal} we provide a method for analyzing perturbations of a periodic orbit in terms of derivatives of this collision map, and use this to prove Theorem \ref{mainTh}. All calculations necessary for this are proven in Section \ref{calc}.

\section{The Collision Map and Its Derivatives} \label{SecMap}

To analyze a periodic orbit, we let $N$ be the number of collision points per period of the trajectory and denote the collision points by $Q_k$, for $ k = 0, 1, \cdots, N-1 $, where the points are ordered consecutively in the path of a particle in the trajectory. \keyterm{Scatterer $k$} will refer to the scatterer containing $Q_k$ on its boundary, whose center we denote by $C_k$. Note that the scatterers, and therefore the points $C_k$, may coincide for multiple values of $k$, except for $C_0$ by the hypothesis of Theorem \ref{mainTh}. \keyterm{Segment $k$} will refer to the portion of the trajectory between collision points $Q_k$ and $Q_{k+1}$.

We will analyze the billiard using the collision map, or the map from the position and velocity at one collision point to the position and velocity at the next collision point. For any $ 0 \leq k \leq N-1 $, we represent this map as $ F_k : (\alpha_k, \phi_k) \mapsto (\alpha_{k+1}, \phi_{k+1}) $, where $ \phi_k \in \R / 2 \pi \mathbb{Z} $ is such that $ Q_k = C_k + (\cos \phi_k, \sin \phi_k) $, and $ \alpha_k \in \bracket{ -\frac{\pi}{2}, \frac{\pi}{2} } $ is the signed angle of incidence at this collision. Some instances of these parameters, along with $\omega_k$ and $r_k$ to be defined below, are illustrated in Figure \figParam.
Where appropriate, we identify the values $C_0$, $Q_0$, $\alpha_0$, and $\phi_0$ with $C_N$, $Q_N$, $\alpha_N$, and $\phi_N$, respectively, due to the periodicity of the orbit.
It can be proven by a simple geometric argument that for all $j$,
\begin{equation} \label{angId}
    \omega_j = \phi_j - \alpha_j \text{ and } \omega_{j-1} = \phi_j + \alpha_j + \pi,
\end{equation}
where all computations on $\phi_i$ or $\omega_i$ are to be performed in $ \R / 2\pi\Z $.

\begin{figure} \label{figParam}
    \includegraphics[width=84mm]{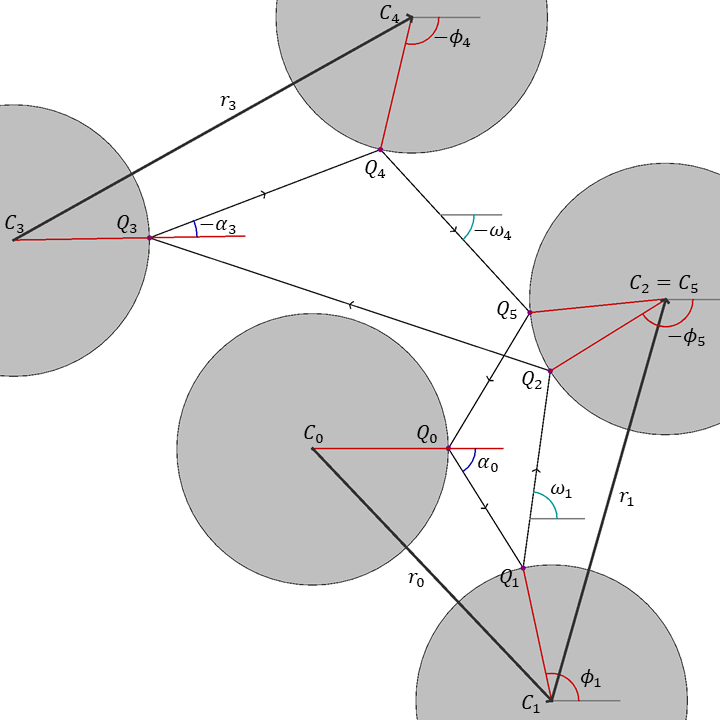}
    \caption{Geometric illustration of the meaning of $r_k$, $C_k$, $Q_k$, $\phi_k$, $\alpha_k$, and $\omega_k$. Collision points $Q_2$ and $Q_5$ lie on the same scatterer, so we set $ C_2 = C_5 $.}
\end{figure}


Chernov and Makarin (2006) give a general form for the derivative of the collision map in terms of the curvature of scatterers, angles of incidence, and distance between collision points. In our case, the curvature of every scatterer is $1$, so as a special case of this general form,
\begin{equation} \label{mapderiveq}
    \mapderiv{k} = -\sec \alpha_{k+1} \mat{s_k + \cos \alpha_{k+1}}{-(s_k + \cos \alpha_k + \cos \alpha_{k+1})}{-s_k}{s_k + \cos \alpha_k},
\end{equation}
To make our calculations simpler, we apply an affine transformation $ u_k = (\phi_k - \alpha_k, \phi_k + \alpha_k + \pi) $ and analyze the map $ B_k : u_k \mapsto u_{k+1} $. Then $ u_k = (\omega_k, \omega_{k-1}) $, where for each segment $i$ in the trajectory, a particle moves in the direction of $ (\cos \omega_i, \sin \omega_i) $. Note that this is a conjugacy except that it fails to be surjective when a grazing collision occurs at $Q_k$. However, at all non-singular points, the derivative can be computed as a conjugacy from (\ref{mapderiveq}):
\begin{equation} \label{mapderivueq}
    \mapuderiv{k} = -\sec \alpha_{k+1} \mat{2s_k + \cos \alpha_k + \cos \alpha_{k+1}}{\cos \alpha_k}{-\cos \alpha_{k+1}}{0}.
\end{equation}

Note that the above collision maps are also dependent on $C_k$ and $C_{k+1}$. It is possible to define a single collision map $B$ with an additional parameter $r$ such that for all $k$, $ B(u_k,r_k) = u_{k+1} $, where $ r_k = C_{k+1} - C_k $ is the displacement between the centers of the respective scatterers. Then, $ \pderiv{B}{u}(u_k,r_k) $ is as in (\ref{mapderivueq}) and
\begin{equation} \label{dBdr}
    \pderiv{B}{r}(u_k,r_k) = -\sec \alpha_{k+1} \mat{2 \sin \omega_k}{-2 \cos \omega_k}{0}{0}.
\end{equation}
To prove this, we apply a geometric argument and break $dr$ up into components of $r_k$. The case where $dr$ is on the $y$-axis is illustrated in Figure \figdBdr. Note that we move scatterer $k$, which moves in the opposite direction as $r_k$ moves. The trajectory near scatterer $k$ moves only in translation with the scatterer, since $u_k$ remains fixed. This means that $ dQ_k = dC_k $ and $\omega_k$ is constant, so the second row of (\ref{dBdr}) is zero.
We can compute the change $d\ell$ in the ``position'' of segment $k$ as a projection of $dC_k$ onto its normal.
$dQ_{k+1}$ must also have the same such projection and satisfy the identity
\begin{equation}
    dQ_{k+1} = d\phi_{k+1} (-\sin \phi_{k+1}, \cos \phi_{k+1}),
\end{equation}
from which $d\phi_{k+1}$ can be computed. Finally, as a consequence of (\ref{angId}), we have the identity $ 2\phi_{k+1} = \omega_k + \omega_{k+1} + \pi $, and hence $ d\omega_{k+1} = 2 d\phi_{k+1} $.

\begin{figure} \label{figdBdr}
  \includegraphics[width=84mm]{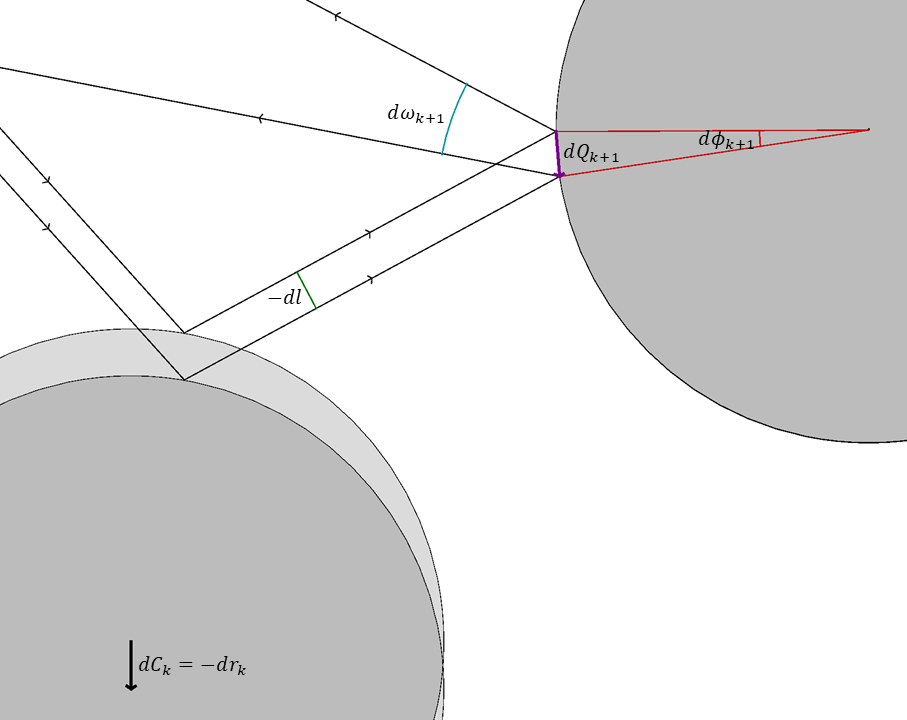}
\caption{The computation of the $y$-component of $\pderiv{B}{r}(u_k,r_k)$}
\end{figure}

\section{Analyzing Perturbations of a Periodic Orbit} \label{SecAnal}

We write the derivative with respect to $u$ of multiple iterations of the collision map as
\begin{equation} \label{ukujdef}
    \deriv{u_k}{u_j} := \pderiv{B}{u}(u_{k-1},r_{k-1}) \cdots \pderiv{B}{u}(u_{j+1},r_{j+1}) \pderiv{B}{u}(u_j,r_j).
\end{equation}
We derive a formula for $\deriv{u_k}{u_j}$ and perform all other applicable calculations in Section \ref{calc}. In this section, we describe our general method.

To analyze perturbations of a periodic orbit and a periodic trajectory, we let $N$ be the number of collision points per period and impose the condition that $ u_N = u_0 $, which is equivalent to
\begin{equation} \label{periodCond}
    u_0 = B\parens{\cdots\parens{B\parens{B\parens{B(u_0,r_0),r_1},r_2}\cdots},r_{N-1}}.
\end{equation}
We analyze a family of billiard systems, along with their associated periodic orbits, where the system and trajectory vary continuously with some parameter $\gamma$, and $ \gamma = 0 $ corresponds to the initial, unperturbed system.
We also consider all parameters of the trajectory and collision points as functions of $\gamma$ and write $\alpha_k(\gamma)$, $\omega_k(\gamma)$, $s_k(\gamma)$, $Q_k(\gamma)$, etc., to denote the respective values in the trajectory and system corresponding to a particular $\gamma$.
We move the position of scatterer 0 with $\gamma$ and leave the positions of the other scatterers unchanged. We are assuming that $Q_0$ is the only collision point on this scatterer, so then $ r'_{N-1}(\gamma) = -r'_0(\gamma) $, and for all $ 1 \leq i \leq N-2 $, $ r'_i(\gamma) = 0 $. Therefore, differentiating (\ref{periodCond}) with respect to $u_0$ yields
\begin{equation} \label{periodDerivEq}
    u_0'(\gamma) = \deriv{u_N}{u_0} u_0'(\gamma) + \parens{ \deriv{u_N}{u_1} \pderiv{B}{r}(u_0,r_0) - \pderiv{B}{r}(u_{N-1},r_{N-1}) } r_0'(\gamma).
\end{equation}
We solve this equation in Section \ref{calc}.

If the periodic orbit corresponding to any value of $\gamma$ contains a grazing collision before scatterer $0$ reaches segment $k$, then there must be a singularity somewhere in the trajectory, and the goal of Theorem \ref{mainTh} would already be attained. Therefore, in all future analysis and calculations, we assume this is not the case. As a consequence, the sequence of points $ \set{C_0,C_1,\cdots,C_{N-1}} $ is independent of $\gamma$, except for the fact that $C_0$ depends continuously on $ \gamma $.

As our analysis is invariant under isometries of the entire system, we assume without loss of generality that $C_0$ is the origin and $ \phi_0 = 0 $.
We also assume that $ h(0) $ is positive. The case where $ h(0) $ is negative follows by reflecting the entire system and the periodic trajectory about the $x$-axis.

To move scatterer $0$ toward segment $k$ of the trajectory, we let $Z$ be the point on segment $k$ closest to scatterer $0$, or equivalently, closest to its center $C_0$, and $h$ be the signed distance between $Z$ and $C_0$, such that $ Z = (h \sin \omega_k, -h \cos \omega_k) $ (This is possible since $ (-\sin \omega_k, \cos \omega_k) $ is the normal to segment $k$).
By the assumptions of Theorem \ref{mainTh}, $ 1 < \abs{h} < 1 + \eps $, but $h$ could be positive or negative.

We consider $C_0$ as a unit-speed function of $\gamma$ and set $ C_0'(\gamma) = (\cos \theta, \sin \theta) $, where $ \theta = \omega_k \pm \frac{\pi}{2} $ is such that $C_0'(\gamma)$ is directed toward $Z$, as shown in Figure \figCond.
To prove Theorem \ref{mainTh}, we aim to decrease the distance between $C_0$ and segment $k$ of the trajectory to $1$. This distance is influenced my the motion of scatterer $0$ as well as the movement of segment $k$ itself near point $Z$. Therefore, we need a notion of describing how the ``position'' of the trajectory at a specific point varies with changes in the system.

For this purpose, we introduce the concept of a \keyterm{wave front}, as defined in Section 3.7 of a book by Chernov and Makarin (2006). Let $P$ be any point on a segment $k$ of the trajectory other than a collision point, for some value $\gamma_0$ of $\gamma$. Then, there exists a smooth curve $\sigma_P$, parametrized by $\gamma$ on some open interval containing $\gamma_0$ such that $ \sigma_P(\gamma_0) = P $, and for all $\gamma$, segment $k$ of the trajectory with parameter $\gamma$ passes through $\sigma_P(\gamma)$ and is perpendicular to $\sigma_P$. The curve $\sigma_P$, together with the family of trajectories for various values of $\gamma$, is a \keyterm{wave front}, as shown in Figure \figWave.

For any point $P$ on a segment $k$ of the trajectory, there is some constant $\ell(P,\gamma_0)$ such that $ \sigma'_P(\gamma_0) = \ell(P,\gamma_0) (-\sin \omega_k(\gamma_0), \cos \omega_k(\gamma_0)) $. Then, $\ell$ represents the magnitude of the derivative of the ``position'' of the trajectory at point $P$ with respect to $\gamma$. Since we assume $h$ is positive and $C_0$ is moving toward $Z$ at unit speed, we have $ h'(\gamma) = \ell(Z,\gamma) - 1 $.
To prove Theorem \ref{mainTh}, we prove that for some positive constant $L$, $ h'_0(\gamma) < -L $ for sufficient values of $\gamma$.

It is evident that as $P$ moves along the trajectory and $\gamma$ remains fixed, the only possible discontinuities of $\ell(P,\gamma)$ are at collision points.
To analyze the behavior of $\ell(P,\gamma)$ at the collision points, we define
\begin{equation} \label{lkdef}
    \ell_k^+(\gamma) := \lim_{P \rightarrow Q_k^+} \ell(P,\gamma) \text{ and } \ell_k^-(\gamma) := \lim_{P \rightarrow Q_k^-} \ell(P,\gamma),
\end{equation}
where the limits are taken as $P$ varies along the trajectory, approaching from immediately after the collision and immediately before the collision, respectively.
The wave fronts at two different points, as well as these limits, are shown in Figure \figWave.
$ \ell_k^- $ and $ \ell_k^+ $ are the derivatives with respect to $\gamma$ of the ``positions'' of segments $k-1$ and $k$, respectively, of the trajectory near the collision point $Q_k$. These values are illustrated in Figure \figWave.

\begin{figure} \label{figWave}
  \includegraphics[width=84mm]{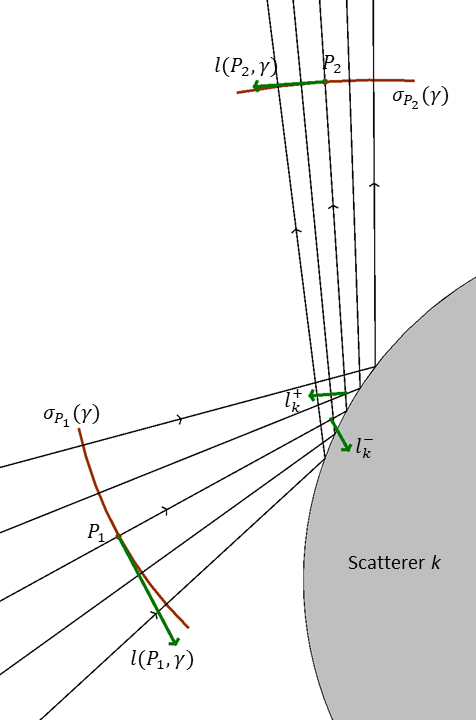}
\caption{Two points on the trajectory, $P_1$ and $P_2$, their respective wave fronts, and tangent vectors corresponding to $\ell(P_i,\gamma)$. Note that in this case, $\ell(P_1,\gamma)$ is negative and $\ell(P_2,\gamma)$ is positive. The values $\ell_k^+$ and $\ell_k^-$ are the one-sided limits of the values $\ell(P,\gamma)$ as $P$ approaches the collision point with scatterer $k$.}
\end{figure}

It can be proven through a simple geometric argument, as shown in Figure \figlfixed, that if $Q_k$ is independent of $\gamma$, then
\begin{equation} \label{lfixed}
    \ell_k^+(\gamma) = -\ell_k^-(\gamma) = \cos \alpha_k(\gamma) \phi_k'(\gamma).
\end{equation}
This is the case for all collision points except $Q_0$, so $\abs{\ell(P,\gamma)}$ varies continuously with $P$ everywhere except at $Q_0$. We use a property of dispersing billiards to prove the following lemma:

\begin{figure} \label{figlfixed}
  \includegraphics[width=84mm]{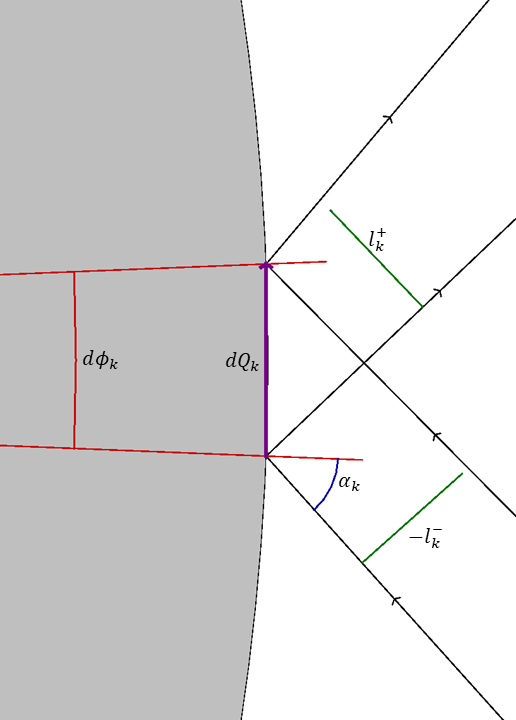}
\caption{The proof of (\ref{lfixed}). $\ell_k^-$ and $\ell_k^+$ are projetions of $dQ_k$ onto the normal vectors to segments $k-1$ and $k$, respectively, of the trajectory. Therefore, these values can be computed by $dQ_k$ and the angle of incidence $\alpha_k$.}
\end{figure}

\begin{lemma} \label{lklem}
    For all points $P$ in the trajectory, $ \abs{\ell(P,\gamma)} \leq \max \set{\abs{\ell_0^+(\gamma)}, \abs{\ell_0^-(\gamma)}} $.
\end{lemma}

\begin{proof}
    We use the concept of dispersing and focusing wave fronts as defined by Chernov and Makarin (2006): a wave front $\sigma$ is \keyterm{dispersing} precisely when $\abs{\ell(P,\gamma)}$ is increasing as $P$ moves along the trajectory in the forward direction, and $\sigma$ is \keyterm{focusing} when $\abs{\ell(P,\gamma)}$ decreases. Chernov and Makarin also prove that any dispersing wave front in a dispersing billiard will remain dispersing in the future, as long as they are unaffected by changes in the system. This means that, as $P$ moves along the trajectory in one period from $Q_0$ back to itself and $\gamma$ remains constant, there are three possibilities:
    \begin{itemize}
        \item $\abs{\ell(P,\gamma)}$ is monotonically decreasing,
        \item $\abs{\ell(P,\gamma)}$ is monotonically increasing, or
        \item $\abs{\ell(P,\gamma)}$ is monotonically decreasing from $Q_0$ to some point $E$ on the trajectory, and monotonically increasing from $E$ to $Q_0$.
    \end{itemize}
    In all of these cases, the lemma is true.
\end{proof}

Since $ h'(\gamma) = \ell(Z,\gamma) - 1 $, the following is a direct consequence:

\begin{corollary}
    It holds that
    \begin{equation} \label{hprimebound}
        h'(\gamma) < -\min \set{1 - \ell_0^+, 1 - \ell_0^-}.
    \end{equation}
\end{corollary}

We are now ready to prove our main theorem.

\begin{proof}[Theorem \ref{mainTh}]
    It suffices to prove that for any $ \eps > 0 $, there exists $ L > 0 $ and $ \delta > 0 $ such that $ \frac{\delta}{L} < \eps $, and for any $ \gamma \in \clop{0, \frac{\delta}{L}} $ such that there is no singularity in the trajectory, $ h'(\gamma) < -L $.

    We prove in Section \ref{calc} that $ 1 - \ell_0^+ $ and $ 1 - \ell_0^- $ are both greater than
    $$ \frac{m \cos \alpha_0}{(M+1)(2m+1)}, $$
    where $m$ is the minimum separation between two scatterers and $M$ is the maximum distance between centers of two scatterers.
    We assume for now that $ \abs{\alpha_0(\gamma)} < \frac{\pi}{6} $ for all $\gamma$ and justify this later. Under this assumption,
    $$ \frac{m \cos \alpha_0}{(M+1)(2m+1)} > \frac{\sqrt{3}m}{2(M+1)(2m+1)}. $$
    The bounds $m$ and $M$ may change with $\gamma$, but their change is restricted by the distance in which scatterer $0$ moves: for all values of $\gamma$, $ m \geq m_0 - \gamma $ and $ M \leq M_0 + \gamma $, where $m_0$ and $M_0$ are the values of $m$ and $M$, respectively, for the unperturbed system corresponding to $ \gamma = 0 $.
    Therefore, it is possible to choose $ \delta' < \eps $ sufficiently small so that $m$ and $M$ remain bounded by positive constants, and $L$ sufficiently small so that whenever $ \gamma < \delta' $, $ h'(\gamma) < -L $. This will satisfy the condition of Theorem \ref{mainTh} with $ \delta = L \delta' $.
    
    We now prove that $ \abs{\alpha_0} < \frac{\pi}{6} $ holds for all $ \gamma \in [0, \delta') $. By assumption, it holds for $ \gamma = 0 $. We prove in Section \ref{calc} that if the minimum separation between two disks is less than $m$, then $ \abs{\alpha_0'(\gamma)} < \frac{3}{m} $, which is bounded by some constant $C$ for $ \gamma < \delta' $.
    Therefore, for all $ \gamma < \delta' $, $ \abs{\alpha_0(\gamma)} < \abs{\alpha_0(0)} + C \gamma $. By choosing $\delta'$ sufficiently small (which now must depend on $\alpha_0$), we can ensure that $ \abs{\alpha_0(\gamma)} < \frac{\pi}{6} $ for all $ \gamma < \delta' $, and thus the proof of the theorem is complete.
\end{proof}

\section{Conclusion}

We believe that Theorem \ref{mainTh} is an important step in proving the conjecture of Turaev and Rom-Kedar, which together with their results, would imply that every dispersing billiard system can be transformed into one with regions of stability through arbitrarily small perturbations. To solve this conjecture would require removing the second bulleted condition of Theorem \ref{mainTh}, the restriction that scatterer $0$ only has one collision point per period.

With this restriction removed, the proof of Lemma \ref{lklem} would no longer be valid and (\ref{periodDerivEq}) would take a more complex form, as the other collisions with scatterer $0$ must be incorporated. Therefore, to removing this restriction will require new methods.
If a periodic orbit in a dispersing billiard system can be constructed where $ \abs{\ell(Z,\gamma)} \geq 1 $, then the approach of moving only scatterer $0$ will certainly fail.
Therefore, to generalize Theorem \ref{mainTh} by removing this restriction, it may be necessary to allow not only scatterer $0$ but also other scatterers to be moved. Then, it would suffice to prove that for some combination of functions $C_k(\gamma)$ representing the positions of \textit{all} scatterers, the respective value of $ \abs{h'(\gamma)} $ is bounded by some constant.

\appendix

\section{Calculations} \label{calc}

We first define a recursive function that appears in calculating the derivatives of iterations of the collision map. These functions will be used extensively in the calculations to follow.
Specifically, we define
\begin{equation} \label{Gdef}
    \Gxx(j,k) := \begin{cases} 0, & k = j - 1 \\ 1, & k = j \\ (2s_{k-1} + \cos \alpha_{k-1} + \cos \alpha_k) \Gxx(j,k-1) - \cos^2 \alpha_{k-1} \Gxx(j,k-2), & k \geq j + 1, \end{cases}
\end{equation}
and 
\begin{equation} \label{pcosdef}
    \pcos(j,k) := \cos \alpha_j \cos \alpha_{j+1} \cdots \cos \alpha_{k-1}.
\end{equation}
Before beginning our calculations, we introduce some useful identities associated with the $G$ function.

\begin{lemma} \label{Gprodreduce}
	For any $ j \leq k $, $ \Gxx(j,k-1) \Gxx(j+1,k) - \Gxx(j,k) \Gxx(j+1,k-1) = [\pcos(j+1,k)]^2 $.
\end{lemma}

\begin{proof}
	The determinant of the Jacobian of the collision map is $ \det \pderiv{B}{u}(u_k,r_k) = \frac{\cos \alpha_k}{\cos \alpha_{k+1}} $, so it follows that $ \det \deriv{u_k}{u_j} = \frac{\cos \alpha_j}{\cos \alpha_k} $, where $ \deriv{u_k}{u_j} $ is as defined in (\ref{ukujdef}). Applying (\ref{mapderivueq}) yields
	$$ \det \deriv{u_k}{u_j} = \frac{\cos \alpha_j \cos \alpha_k}{[\pcos(j+1,k+1)]^2} [\Gxx(j,k-1) \Gxx(j+1,k) - \Gxx(j,k) \Gxx(j+1,k-1)]. $$
	Equating these two values and simplifying the result completes the proof.
\end{proof}

We also have the following generalization of the recursive formula given in the definition:

\begin{lemma} \label{Gpart}
	For any $i,j,k$ such that $i \leq j \leq k$, $ \Gxx(i,k) = \Gxx(i,j) \Gxx(j,k) - \cos^2 \alpha_j \Gxx(i,j-1) \Gxx(j+1,k) $.
\end{lemma}

In particular, for $ j < k $,
\begin{equation} \label{Gdefflip}
    \Gxx(j,k) = (2s_j + \cos \alpha_j + \cos \alpha_{j+1}) \Gxx(j+1,k) - \cos^2 \alpha_{j+1} \Gxx(j+2,k).
\end{equation}
This can be used as an alternative definition of $\Gxx$ and is related to the original definition by the billiard involution.

\begin{proof}
	The result can be directly verified for the cases $k=j$ and $k=j+1$. Assume the statement is true for $(i,j,k-1)$ and $(i,j,k-2)$. Then,
	\begin{equation} \begin{split}
    	\Gxx(i,k) & = (2s_{k-1} + \cos \alpha_{k-1} + \cos \alpha_k) [\Gxx(i,j) \Gxx(j,k-1) - \cos^2 \alpha_j \Gxx(i,j-1) \Gxx(j+1,k-1)] \\
    	& \qquad - \cos^2 \alpha_{k-1} [\Gxx(i,j) \Gxx(j,k-2) - \cos^2 \alpha_j \Gxx(i,j-1) \Gxx(j+1,k-2)] \\
	    & = \Gxx(i,j) [ (2s_{k-1} + \cos \alpha_{k-1} + \cos \alpha_k) \Gxx(j,k-1) - \cos^2 \alpha_{k-1} \Gxx(j,k-2) ] \\
	    & \qquad + \cos^2 \alpha_j \Gxx(i,j-1) [ (2s_{k-1} + \cos \alpha_{k-1} + \cos \alpha_k) \Gxx(j+1,k-1) - \cos^2 \alpha_{k-1} \Gxx(j+1,k-2) ] \\
	    & = \Gxx(i,j) \Gxx(j,k) - \cos^2 \alpha_j \Gxx(i,j-1) \Gxx(j+1,k)
	\end{split} \end{equation}
\end{proof}

We proceed with our calculations necessary to prove Theorem \ref{mainTh}.

\begin{proposition} \label{derivformulaProp}
    For all integers $j$ and $k$, where $ j < k $,
	$$ \deriv{u_k}{u_j} = \frac{(-1)^{k-j}}{\pcos(j+1,k+1)} \mat{\Gxx(j,k)}{\cos \alpha_j \Gxx(j+1,k)}{-\cos \alpha_k \Gxx(j,k-1)}{-\cos \alpha_j \cos \alpha_k \Gxx(j+1,k-1)}. $$
\end{proposition}

\begin{proof}
	The case where $k=j+1$ can be directly verified using (\ref{mapderivueq}). Assume the statement is true for some $j$ and $k$. The inductive result can be proven by computing the product
	$$ \deriv{u_{k+1}}{u_j} = \deriv{u_{k+1}}{u_k} \deriv{u_k}{u_j} $$
	and applying the recursive definition of $G$.
\end{proof}

We now introduce some new definitions to simplify our calculations:
\begin{itemize}
    \item $ \dpr := \Gxx(0,N) - \cos^2 \alpha_0 \Gxx(1,N-1) + 2 \ppr(0,N). $
    \item $ \Gnx(j,k) := \Gxx(j,k) - \cos \alpha_j \Gxx(j+1,k). $
    \item $ \Gxn(j,k) := \Gxx(j,k) - \cos \alpha_k \Gxx(j,k-1). $
    \item $ \Gnx(j,k) := \Gxx(j,k) + \cos \alpha_j \Gxx(j+1,k). $
    \item $ \Gxn(j,k) := \Gxx(j,k) + \cos \alpha_k \Gxx(j,k-1). $
    \item $ \ppr(j,k) := (-1)^{N+1} \pcos(j,k). $
\end{itemize}

\begin{proposition} \label{periodCondSol}
    The solution to (\ref{periodDerivEq}) under the assumption $ r_0'(\gamma) = -(\cos \theta, \sin \theta) $ is
    \begin{equation}
        u_0'(\gamma) = -\frac{2}{\dpr} \mat{\Gxx(1,N)}{\cos \alpha_0 \Gxx(1,N-1) - \ppr(1,N)}{-\cos \alpha_0 \Gxx(1,N-1) + \ppr(1,N)}{-\Gxx(0,N-1)} \cvec{\sin(\alpha_0 + \theta)}{\sin(\alpha_0 - \theta)}.
    \end{equation}
\end{proposition}

\begin{proof}
    (\ref{periodDerivEq}) can be rewritten as
    \begin{equation} \label{ureq}
        u_0'(\gamma) = \parens{I - \deriv{u_N}{u_0}}^{-1} \parens{ \deriv{u_N}{u_1} \pderiv{B}{r}(u_0,r_0) - \pderiv{B}{r}(u_{N-1},r_{N-1}) } r_0'(\gamma),
    \end{equation}
    where $\frac{du_N}{du_0}$ is as defined in (\ref{ukujdef}). By Proposition \ref{derivformulaProp},
    $$ I - \frac{du_N}{du_0} = \frac{1}{\ppr(0,N)} \mat{\Gxx(0,N) + \ppr(0,N)}{\cos \alpha_0 \Gxx(1,N)}{-\cos \alpha_0 \Gxx(0,N-1)}{-\cos^2 \alpha_0 \Gxx(1,N-1) + \ppr(0,N)}. $$
    Using Lemma \ref{Gprodreduce} and the formula for the determinant of a $ 2 \times 2 $ matrix, it can be shown that
    $$ \det \mat{\Gxx(0,N) + \ppr(0,N)}{\cos \alpha_0 \Gxx(1,N)}{-\cos \alpha_0 \Gxx(0,N-1)}{-\cos^2 \alpha_0 \Gxx(1,N-1) + \ppr(0,N)} = \ppr(0,N) \dpr, $$
    so
    $$ \parens{I - \frac{du_N}{du_0}}^{-1} = \frac{1}{\dpr} \mat{-\cos^2 \alpha_0 \Gxx(1,N-1) + \ppr(0,N)}{-\cos \alpha_0 \Gxx(1,N)}{\cos \alpha_0 \Gxx(0,N-1)}{G(0,N) + \ppr(0,N)}. $$
    
    By (\ref{angId}) and the assumption that $ \phi_0 = 0 $, $ \omega_0 = \alpha_0 $ and $ \omega_{N-1} = \pi + \alpha_0 $. Using these values, (\ref{dBdr}), and Proposition \ref{derivformulaProp},
    $$ \deriv{u_N}{u_1} \pderiv{B}{r}(u_0,r_0) - \pderiv{B}{r}(u_{N-1},r_{N-1}) = \frac{2}{\ppr(0,N)} \mat{\Gxx(1,N) - \ppr(1,N)}{\Gxx(1,N) + \ppr(1,N)}{-\cos \alpha_0 \Gxx(1,N-1)}{-\cos \alpha_0 \Gxx(1,N-1)} \mat{\sin \alpha_0}{0}{0}{\cos \alpha_0}. $$
    
    Multiplying the above two values, using Lemma \ref{Gprodreduce}, and taking a factor of $ \cos \alpha_0 $ out of $\ppr(0,N)$ where appropriate yields
     \begin{equation} \label{u0intermsofr0}
        u_0'(\gamma) = \frac{2}{\dpr} \mat{\Gxp(1,N) - \ppr(1,N)}{\Gxn(1,N) + \ppr(1,N)}{-\Gpx(0,N-1) + \ppr(1,N)}{\Gnx(0,N-1) + \ppr(1,N)} \mat{\sin \alpha_0}{0}{0}{\cos \alpha_0} r_0'(\gamma).
    \end{equation}
    We assume that $ r_0'(\gamma_0) = -(\cos \theta, \sin \theta) $, so applying the identity
    $$ \mat{\sin \alpha_0}{0}{0}{\cos \alpha_0} r_0'(\gamma) = \frac{1}{2} \mat{1}{1}{1}{-1} \cvec{\sin(\alpha_0 + \theta)}{\sin(\alpha_0 - \theta)}, $$
    to (\ref{u0intermsofr0}) and simplifying completes the proof.
\end{proof}

\begin{proposition} \label{l0values}
    The values $ \ell_0^+(\gamma) $ and $ \ell_0^-(\gamma) $, as defined in (\ref{lkdef}), are
    \begin{equation} \label{l0plus}
        \ell_0^+(\gamma) = \frac{\parens{\Gnx(0,N) + \ppr(0,N)} \sin(\alpha_0 + \theta) + \parens{\cos \alpha_0 \Gnx(0,N-1) + \ppr(0,N)} \sin(\alpha_0 - \theta)}{\dpr}
    \end{equation}
    and
    \begin{equation} \label{l0minus}
        \ell_0^-(\gamma) = \frac{\parens{\cos \alpha_0 \Gxn(1,N) + \ppr(0,N)} \sin(\alpha_0 + \theta) + \parens{\Gxn(0,N) + \ppr(0,N)} \sin(\alpha_0 - \theta)}{\dpr}.
    \end{equation}
\end{proposition}

\begin{proof}
    We only prove the result for $\ell_0^+$, as the proof for $\ell_0^-$ is similar.
    
    Note that in the beginning of Section \ref{SecAnal}, $\ell_0^+$ is defined to be the limit of the derivatives $ \sigma'_P(\gamma) $, where $ P \rightarrow Q_0^+ $. By considering the trajectory itself as a function of $r_0$ and $u_0$ for a fixed $P$, we can write
    $$ \sigma'_P(\gamma) = \pderiv{\sigma_P}{u_0} u_0'(\gamma) + \pderiv{\sigma_P}{r_0} r_0'(\gamma). $$
    This implies that we can compute $\ell_0^+$ by a similar decomposition into partial derivatives of the form
    \begin{equation} \label{l0pAB}
        \ell_0^+ = A u_0'(\gamma) + B r_0'(\gamma),
    \end{equation}
    and compute $A$ and $B$ by considering the effect of changes in $u_0$ and $r_0$, respectively, on the trajectory near $Q_0$.
    
    In the case where $r_0$ is constant, $\ell_0^+$ can be calculated by (\ref{lfixed}). Applying (\ref{angId}) and the identity $ dQ_0 = (0, d\phi_0) $, we have
    $$ A = \rvec{\sin \alpha_0}{\cos \alpha_0} \pderiv{Q_0}{u_0} = \frac{\cos \alpha_0}{2} \rvec{1}{1}. $$

    If $u_0$ is constant, then the trajectory near $Q_0$ moves only in translation with $C_0$, as shown in Figure \figlzero. Therefore, $B$ is a projection of $ C_0'(\gamma) $ onto the unit normal $ (\sin \omega_0, -\cos \omega_0) $. Since $ r_0'(\gamma) = -C_0'(\gamma) $ and $ \alpha_0 = -\omega_0 $, this means
    $$ B = \rvec{-\sin \alpha_0}{-\cos \alpha_0}. $$
    Using the result $ r_0'(\gamma) = -(\cos \theta, \sin \theta) $, we have $ B r_0'(\gamma) = \sin(\alpha_0 + \theta) $. Applying the values of $A$, $B$, and $u'_0(\gamma)$ as computed in Proposition \ref{periodCondSol} to (\ref{l0pAB}) completes the proof.
\end{proof}

\begin{figure} \label{figl0}
  \includegraphics[width=84mm]{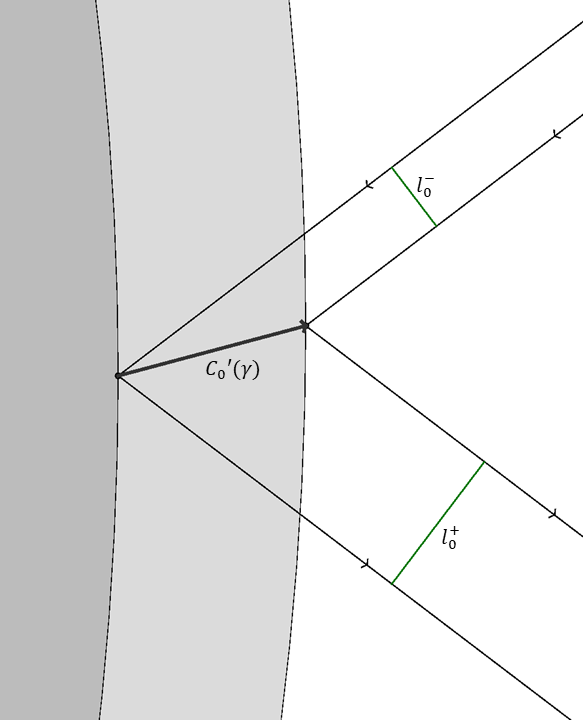}
\caption{If $u_0$ does not change with $\gamma$, then the trajectory near scatterer $0$ moves only in translation with the scatterer. Therefore, $ dC_0 = dQ_0 $, and the values $\ell_0^-$ and $\ell_0^+$ can be computed simply as projections of $dC_0$ onto the normal vectors to the respective segments of the trajectory.}
\end{figure}

We now present some inequalities that we will use to bound the appropriate quantities:
\begin{lemma} \label{Gineq}
    For any $j$ and $k$ where $ j < k $,
    \begin{enumerate}[(a)]
        \item $ \parens{2s_{k-1} + \cos \alpha_k} \Gxx(j,k-1) < \Gxx(j,k) < \parens{2s_{k-1} + \cos \alpha_{k-1} + \cos \alpha_k} \Gxx(j,k-1) $
        \item $ \cos \alpha_k \Gxx(j,k-1) < \Gxx(j,k) $
        \item $ \Gxx(j,k) > \pcos(j+1,k+1) $
        \item $ \parens{2s_j + \cos \alpha_j} \Gxx(j+1,k) < \Gxx(j,k) < \parens{2s_j + \cos \alpha_j + \cos \alpha_{j+1}} \Gxx(j+1,k) $
        \item $ \cos \alpha_j \Gxx(j+1,k) < \Gxx(j,k) $
        \item $ \Gxx(j,k) > \pcos(j,k) $
    \end{enumerate}
\end{lemma}

\begin{proof}
    (b) is simply a stronger condition than (a), and (a) and (b) can both be proven inductively from the recursive definition. (c) follows from repeated application of (b). (d)-(f) are related to (a)-(c) by the billiard involution, and can be similarly proven using (\ref{Gdefflip}).
\end{proof}

\begin{proposition}
    The quantities $ \abs{\ell_0^+} $ and $ \abs{\ell_0^-} $ are both less than
    $$ 1 - \frac{m \cos \alpha_0}{(M+1)(2m+1)}, $$
    where $M$ is the maximum distance between the centers of any two scatterers and $m$ is the minimum separation between any two scatterers.
\end{proposition}

\begin{proof}
    Note that the lengths of all segments $s_i$ satisfy $ m \leq s_i \leq M $.

    By the formulas given in Proposition \ref{l0values}, it is sufficient to prove that the values
    $$ \frac{\Gnx(0,N) + \ppr(0,N)}{\dpr}, \frac{\cos \alpha_0 \Gnx(0,N-1) + \ppr(0,N)}{\dpr}, \frac{\Gxn(0,N) + \ppr(0,N)}{\dpr}, \text{ and } \frac{\cos \alpha_0 \Gxn(1,N) + \ppr(0,N)}{\dpr} $$
    are all less than $ 1 - \frac{\cos \alpha_0}{2M + 1 + \cos \alpha_0} $. We will only prove that
    $$ \frac{\Gnx(0,N) + \ppr(0,N)}{\dpr} < 1 - \frac{\cos \alpha_0}{2M + 1 + \alpha_0}, $$
    as the proof for $ \frac{\Gxn(0,N) + \ppr(0,N)}{\dpr} $ is similar and the others directly follow from Lemma \ref{Gineq}(b) and Lemma \ref{Gineq}(e).
    
    Notice that $ D < \Gxx(0,N) + 2\ppr(0,N) $, so by expanding $D$ and $\Gnx$ in the numerator,
    \begin{equation} \begin{split}
        1 - \frac{\Gnx(0,N) + \ppr(0,N)}{D} & = \frac{\cos \alpha_0 \Gnx(0,N-1) + \ppr(0,N)}{D} \\
        & > \frac{\cos \alpha_0 \Gnx(0,N-1) + \ppr(0,N)}{\Gxx(0,N) + 2\ppr(0,N)} \geq \min \set{\frac{1}{2}, \frac{\cos \alpha_0 \Gnx(0,N-1)}{\Gnx(0,N)}}.
    \end{split} \end{equation}
    By Lemma \ref{Gineq}(a),
    $$ \frac{\Gxx(0,N-1)}{\Gxx(0,N)} > \frac{1}{2s_{N-1} + \cos \alpha_{N-1} + \cos \alpha_0} \geq \frac{1}{2M+2} \text{ and } \frac{\Gnx(0,N-1)}{\Gxx(0,N-1} > \frac{2s_0}{2s_0 + \cos \alpha_0} \geq \frac{2m}{2m+1}. $$
    Therefore,
    $$ \frac{\cos \alpha_0 \Gnx(0,N-1)}{\Gnx(0,N)} > \frac{\cos \Gnx(0,N-1)}{\Gxx(0,N)} > \frac{m \cos \alpha_0}{(M+1)(2m+1)}, $$
    which completes the proof.
\end{proof}

\begin{proposition}
    If the minimum separation between any two scatterers is $m$, then
    $$ \abs{\alpha_0'(\gamma)} < \frac{3}{m}. $$
\end{proposition}

\begin{proof}
    Using the result of Proposition \ref{periodCondSol},
    $$ \alpha_0'(\gamma) = \rvec{\frac{1}{2}}{-\frac{1}{2}} u_0'(\gamma) = -\frac{1}{D} \cvec{\Gxx(1,N) + \cos \alpha_0 \Gxx(1,N-1) - \ppr(1,N)}{\Gxx(0,N-1) + \cos \alpha_0 \Gxx(1,N-1) - \ppr(1,N)} \cdot \cvec{\sin(\alpha_0 + \theta)}{\sin(\alpha_0 - \theta)}. $$
    
    By expanding $\dpr,$ using the recursive relation for $\Gxx(0,N),$ and combining terms,
    \begin{equation} \begin{split}
        \dpr & = (2s_{N-1} + \cos \alpha_{N-1} + \cos \alpha_0) \Gxx(0,N-1) - \cos^2 \alpha_{N-1} \Gxx(0,N-2) - \cos^2 \alpha_0 \Gxx(1,N-1) + 2 \ppr(0,N) \\
        & = 2s_{N-1} \Gxx(0,N-1) + \cos \alpha_0 \Gnx(0,N-1) + \cos \alpha_{N-1} \Gxn(0,N-1) + 2 \ppr(0,N).
    \end{split} \end{equation}
    Since $ \cos \alpha_0 \Gnx(0,N-1) $ and $ \cos \alpha_{N-1} \Gxn(0,N-1) $ are both greater than $ \pcos(0,N) $ and $ \ppr(0,N) \geq -\pcos(0,N) $, this means that $ \dpr > 2s_{N-1} \Gxx(0,N-1) $, so
    $$ \frac{\Gxx(0,N-1)}{\dpr} < \frac{1}{2s_{N-1}} \leq \frac{1}{2m}. $$
    
    It can be similarly shown that $ \frac{\Gxx(1,N)}{\dpr} < \frac{1}{2s_0} \leq \frac{1}{2m} $, and by Lemma \ref{Gineq}(b), $ \frac{\cos \alpha_0 \Gxx(1,N-1)}{\dpr} < \frac{1}{2m} $. By Lemma \ref{Gineq}(f), $ \pcos(1,N) = \abs{\ppr(1,N)} < \Gxx(1,N) $, so $ \frac{\abs{\ppr(1,N)}}{\dpr} < \frac{1}{2m} $. These bounds can be combined to yield $ \alpha_0'(\gamma) < \frac{3}{m} $.
\end{proof}

\section{Acknowledgements}

I would like to thank Dr. Oleg Makarenkov for his assistance in this project.


\begin{thebibliography}{9}

\bibitem{Baldwin}
Baldwin, P. R.:
Soft billiard systems,
Physica D. 29, 321-342 (1988)

\bibitem{BS}
Bunimovich, L. A., Sinai, Y. G.:
On a fundamental theorem in the theory of dispersing billiards,
Mat. Sb. 92 (132), 415-431 (1973)

\bibitem{CM}
Chernov, N., Markarian, R.:
Chaotic Billiards,
American Mathematical Society, Providence (2006)

\bibitem{Knauf}
Knauf, A.:
On soft billiard systems,
Physica D. 36, 259-262 (1989)

\bibitem{Kubo}
Kubo, I.:
Perturbed billiard systems, I. The ergodicity of the motion of a particle in a compound central field,
Nagoya Math. J. 61, 1-57 (1976)

\bibitem{TRK}
Turaev, D., Rom-Kedar, V.:
Elliptic islands appearing in near-ergodic flows,
Nonlinearity 11, 575-600 (1998)

\bibitem{TRK2}
Turaev, D., Rom-Kedar, V.:
Billiards: A singular purturbation of smooth Hamiltonian flows,
Chaos 22(2), 1-34 (2012)

\bibitem{RRKT}
Rapoport, A., Rom-Kedar, V., Turaev, D.:
Approximating multi-dimensional Hamiltonian flows by billiards,
arXiv:nlin/0511071 (2005)

\bibitem{Szasz}
Sz\'{a}sz, D.:
Boltzmann's Ergodic Hypothesis, a Conjecture for Centuries?,
Hard Ball Systems and the Lorenz Gas, 421-446 (2000)

\bibitem{Sinai}
Sinai, Y. G.:
Dynamical systems with elastic reflections: Ergodic properties of scattering billiards,
Russ. Math. Surv. 25, 137-189 (1970)

\bibitem{SC87}
Sinai, Y., Chernov, N. I.:
Ergodic properties of certain systems of two-dimensional discs and three-dimensional balls.
Russ. Math. Surv. 42, 181 (1987)

\end{thebibliography}
\end{document}